\documentclass[12pt]{article}
\usepackage{latexsym,amssymb,amsmath,amsfonts,amsthm,dsfont}
\newtheorem{theorem}{Theorem}

\newtheorem{lemma}{Lemma}

\theoremstyle{definition}
\newtheorem{remark}{Remark}

\def\beq{ \begin{equation} }
\def\eeq{ \end{equation} }
\def\mn{\medskip\noindent}

\def\ep{\varepsilon}

\def\square{\vcenter{\vbox{\hrule height .4pt
  \hbox{\vrule width .4pt height 5pt \kern 5pt
        \vrule width .4pt} \hrule height .4pt}}}

\def\EE{\mathbb{E}}
\def\PP{\mathbb{P}}
\def\RR{\mathbb{R}}
\def\ZZ{\mathbb{Z}}

\def\sqz{\kern -0.2em}

\def\clearp{}

\begin{document}

\title{Latent Voter Model on \\ Locally Tree-Like Random Graphs}
\author{Ran Huo and Rick Durrett \thanks{RD is partially supported by NSF grant DMS 1505215 from the probability program.}
\\
Dept. of Math, Duke U., \\ \small P.O. Box 90320, Durham, NC 27708-0320}
\date{}						

\maketitle

\begin{abstract}
In the latent voter model, which models the spread of a technology through a social network, individuals who have just changed
their choice have a latent period, which is exponential with rate $\lambda$,
during which they will not buy a new device. We study site and edge versions of this model on 
random graphs generated by a configuration model in which the degrees $d(x)$ have $3 \le d(x) \le M$. We show that if the number of
vertices $n \to\infty$ and $\log n \ll \lambda_n \ll n$ then the latent voter model has a quasi-stationary state in which
each opinion has probability $\approx 1/2$ and persists in this state for a time that is $\ge n^m$ for any $m<\infty$. 
Thus, even a very small latent period drastically changes the behavior of the voter model.
\end{abstract}

\section{Introduction}

In this paper we will study the latent voter model introduced in 2009 by Lambiotte, Saramaki, and Blondel
 \cite{LSB}. In this model each individual owns one of two types of technology,
say an iPad or a Microsoft Surface tablet. In the voter model on the $d$-dimensional lattice, individuals at times of a rate one Poisson process
pick a neighbor at random and imitate their opinion. However, in the current interpretation of that model 
it is unlikely that someone who has recently bought a new tablet computer
will replace it, so we introduce latent states $1^*$ and $2^*$ in which individuals will not change their opinion. 
If an individual is in state 1 or 2 we call them active. Letting $f_i$ be the fraction
of neighbors in state $i$ or $i^*$, the dynamics can be formulated as follows
\begin{center}
\begin{tabular}{cc}
$1 \to 2^*$ at rate $f_2$ &\qquad $1^* \to 1$ at rate $\lambda$ \\
$2 \to 1^*$ at rate $f_1$ &\qquad  $2^* \to 2$ at rate $\lambda$
\end{tabular}
\end{center}

In \cite{LSB} the authors showed that if individuals in the population interact equally with all the others then system converges to a limit in which
both technologies have frequency close to 1/2. Here, we will study the system with large $\lambda$, since in this case it is a voter model perturbation
in the sense of Cox, Durrett, and Perkins \cite{CDP}. To explain this, we will construct the system using a graphical representation. 
Suppose first that the system takes place on $\ZZ^d$ and that $d\ge 3$. For each $x \in Z^d$ and nearest neighbor $y$, we have independent Poisson 
processes $T^{x,y}_n$, $n \ge 1$. At each time $t=T^{x,y}_n$ we draw an arrow from $(y,t) \to (x,t)$ to indicate that if the individual at $x$
is active at time $t$ then they will imitate the opinion at $y$. 

To implement the other part of the mechanism, we introduce for each site $x$, a Poisson process $T^x_n$, $n\ge 1$ of ``wake-up dots'' that return the voter to the active 
state. 

\begin{itemize}
  \item 
If there is only one voter arrow between two wake up dots, the result is an ordinary voter event. 
  \item
If between two wake up dots there are voter arrows to $x$ from two different neighbors, an event of probability $O(\lambda^2)$ then the $x$ will change its opinion if and only at least one of the two neighbors has a different opinion. To check this, we note that if the first arrow causes a change then the second one is ignored, while if the first arrow comes from a site with the same opinion as the one at $x$ then there will be a change if and only if the second site has an opinion different from the one at $x$. 
 \item
If $t$ is fixed then at a given site there are $O(\lambda)$ wake-up dots by time $t$. Thus if we want to see the influence of intervals with two voter arrows then we want to run time at rate $\lambda$. The probability of $k$ voter arrows between two wake-up dots is $(1+\lambda)^{-k}$, so in the limit the probability of three of more voter events between two wake-up dots goes to 0 as $\lambda\to\infty$.
\end{itemize}

If we let $\lambda = \ep^{-2}$ and let $n_k(x)$ be the number of neighbors in state $k$ then the rate of flips from $i$ to $j$ in the latent voter model is:
$$
\ep^{-2} c^v_{i,j}(x,\xi) + h_{i,j}(x,\xi) \quad\hbox{where}\quad c^v_{i,j}(x,\xi) = 1_{\{\xi(x)=i\}}\frac{n_j(x)}{2d}
$$ 
If we let $y_1, \ldots y_{2d}$ be an enumeration of the nearest neighbors of $x$, the perturbation is
$$
h_{i,j}(x,\xi) = 1_{\{\xi_t(x)=i\}} \frac{2}{(2d)^2} \sum_{1\le k < \ell \le 2d} 1_{ \{\xi(y_k) = j \hbox{ or } \xi(y_\ell) = j\} } 
$$
If we scale space by $\ep$ then Theorem 1.2 of \cite{CDP} shows that under mild assumptions on the perturbation,
the density of 1's in the rescale particle system converges to 
the solution of the limiting PDE:
\beq
\frac{\partial u}{\partial t} = \frac{1}{2} \Delta u + \phi(u) \quad\hbox{with}\quad 
\phi(u) =\langle h_{2,1}(0,\xi) - h_{1,2}(0,\xi) \rangle_u 
\label{react}
\eeq
and $\langle \cdot \rangle_u$ denotes the expected value with respect to the voter model with density $u$. 

Intuitively, \eqref{react} holds because of a separation of time scales.
The rapid voting means that the configuration near $x$ looks like the voter model equilibrium with density $u(t,x)$.
Later in the paper will show, see \eqref{Zdreact}, that in the case of the latent voter model
$$
\phi(u) = c_du(1-u)(1-2u).
$$
 
If we consider the latent voter model on a torus with $n$ sites and let $\lambda_n\to\infty$ then the system can be analyzed using ideas 
from a recent paper of Cox and Durrett \cite{CD}. Define the density of 1's at time $t$ by
\beq
U^n(t) = \frac{1}{n} \sum_x 1_{\{ \xi_{\lambda t}(x) = 1\} }
\label{dens}
\eeq

\begin{theorem} \label{CDlim}
Suppose $n^{2/d} \ll \lambda_n \ll n$. If $U(0) \to u_0$ then $U(t)$ converges uniformly on compact sets to $U(t)$ the solution of
$$
\frac{du}{dt} = c_du (1-u) (1-2u) \qquad u(0) = u_0
$$
\end{theorem}

\subsection{Random graphs}
 
We will explain the intuition behind Theorem \ref{CDlim} after we state our new result that replaces the torus by a random graph $G_n$ generated by the {\it configuration model}. In this model vertices have degree $k$ with probability $p_k$. To create that graph we assign i.i.d.~degrees $d_i$ to the vertices and condition the sum $d_1 + \cdots + d_n$ to be even, which is a necessary condition for the values to be the degrees of a graph. We attach $d_i$
half-edges to each vertex and then pair the half-edges at random. We will assume that

\mn
(A0) the graph $G_n$ has no self-loops or parallel edges. 

\mn
If $\sum_k k^2p_k < \infty$ then the probability of ${\cal G}$ is bounded away from 0 as as $n \to \infty$. See Theorem 3.1.2 of \cite{RGD}.
The reader can consult Chapter 3 of that reference for more on the configuration model. 

It seems likely that the results we prove here are true under the assumption that the degree distribution has finite second moment, but the presence of vertices of large degrees causes a number of technical problems. To avoid these we will assume:

\mn
(A1) $p_m=0$ for $m > M$, i.e., the degree distribution is bounded.

\mn
In addition, we want a graph that is connected and has random walks with good mixing times, so we will also suppose: 

\mn
(A2) $p_k = 0$ for $k \le 2$. 

\mn
The relevance of (A2) for mixing times will be explained in Section 2.
Assumptions (A0), (A1) and (A2) will be in force throughout the paper.

On graphs that are not regular there are two versions of the voter model. 

\mn
(i) The {\it site version} in which sites change their opinions at rate 1, and imitate
a neighbor chosen at random, 
$$
c^s_{i,j}(x,\xi) = 1_{\{\xi(x)=i \}} \frac{n_j(x)}{d(x)}
$$
where $n_j(x)$ is the number of neighbors of $x$ in state $j$, and $d(x)$ is the degree of $x$.

\mn
(ii) the {\it edge version} in which each neighbor that is different from $x$ causes his opinion to change at rate 1:
$$
c^e_{i,j}(x,\xi) = 1_{\{\xi(x)=i \}} n_j(x)
$$

The site version is perhaps the ``obvious'' generalization of the voter model on $\ZZ^d$, e.g., it is a special case of the general formulation used in Liggett \cite{Lig85}: $x$ imitates $y$ with probability $p(x,y)$, where $p$ is a transition probability. However, the edge version has two special properties. First, in the words of \cite{SEM} ``magnetization is conserved,'' i.e., the number of 1's is a martingale. Second, if we consider the biased version in which after an edge $(x,y)$ is picked a 1 at $x$ always imitate a 2 at $y$ but a 2 at $x$ imitates a 1 at $y$ with probability $\rho<1$ then the probability a single 2 takes over a system that is otherwise all 1 is the same as the probability a simple random walk that jumps up with probability $1/(1+\rho)$ and down with probability $\rho/(1+\rho)$ never hits 0. This observation is due to Maruyama in 1970 \cite{Mar}, but has recently been rediscovered by \cite{LHM}, who call this version of the voter model ``isothermal''.

From our discussion of the graphical representation for latent voter model on $\ZZ^d$ it should be clear that the latent voter model on $G_n$ is a voter model perturbation. If we let $y_1, \ldots y_{d(x)}$ be an enumeration of the neighbors of $x$, then in the site version
$$
h^s_{i,j}(x,\xi) = 1_{\{\xi_t(x)=i\}} \frac{2}{(d(x))^2} \sum_{1\le k < \ell \le d(x)} 1_{ \{\xi(y_k) = j \hbox{ or } \xi(y_\ell) = j\} } 
$$
while in the edge version
$$
h^e_{i,j}(x,\xi) = 1_{\{\xi_t(x)=i\}}\cdot  2 \sum_{1\le k < \ell \le d(x)} 1_{ \{\xi(y_k) = j \hbox{ or } \xi(y_\ell) = j\} } 
$$

\begin{theorem} \label{ODElimit}
Suppose that $\log n \ll \lambda_n \ll n$. If we define the desnity as in \eqref{dens} and $U^n(0)\rightarrow u_0 $ then $U^n(t)$ converges in probability and uniformly on compact sets to $u(t)$, the solution of
\beq
\frac{du}{dt} = c_p u(1-u)(1-2u) \qquad u(0) = u_0.
\label{limODE}
\eeq
where the value of $c_p$ depends on the degree distribution and the version of the voter model.
\end{theorem}

\subsection{Duality}

To explain why Theorems \ref{CDlim} and \ref{ODElimit} are true, we will introduce a dual process that is the key to the analysis. The dual process was first introduced more than 20 years ago by Durrett and Neuhauser \cite{DurNeu}, and is the key to work of Cox, Durrett, and Perkins \cite{CDP}. To do this, we construct the process using a graphical representation that generalizes the one introduced for $\ZZ^d$. For each $x \in Z^d$ and neighbor $y$, we have independent Poisson 
processes $T^{x,y}_n$, $n \ge 1$. At each time $t=T^{x,y}_n$ we draw an arrow from $(y,t) \to (x,t)$ to indicate that if the individual at $x$
is active at time $t$ then they will imitate the opinion at $y$. In the edge case all these processes have rate 1. In the site case $T^{x,y}_n$, $n \ge 0$ has
rate $1/d(x)$. To implement the other part of the mechanism, we have for each site $x$, a rate $\lambda$ Poisson process $T^x_n$, $n\ge 1$ of ``wake-up dots'' that return the voter to the active state. 

To compute the state of $x$ at time $t$ we start with a particle at $x$ at time $t$. To be precise $\zeta^{x,t}_0 = \{x\}$. As we work backwards in time the particle does not move until the first time $s$ there is an arrow $(y,t-s) \to (x,t-s)$. 

\begin{itemize}
  \item 
If this is the only voter arrow between the two adjacent wake-up dots then the particle jumps to $y$. In the edge case the random walk jumps at equal rate to all neighbors so its stationary distribution $\pi$ is uniform. In the site case, the random walk jumps to each neighbor of $x$ with probability $1/d(x)$ so the stationary distribution i $\pi(x) = d(x)/D$ where $D = \sum_y \pi(y)$.
\item
If in the interval between the two adjacent wake-up dots there are arrows from $k$ distinct $y_i$ then the state changes to $\{ x, y_1, \ldots y_k \}$ since we need to know the current state of all these points to know what change should occur in the process. In the limit as $\lambda \to\infty$ we will only see branchings that add two $y_i$. We include the case $k>2$ to have the dual process well-defined.  
\item
We do not need to know the order of the arrows because $x$ will change if at least one of the $y_i$ has a different opinion.  When $\lambda$ is small some of the $y_i$ might change their state during the interval between the two wake-up dots but this possibility has probability zero in the limit.
\item
To complete the definition of the dual, we declare that if a branching event adds a point already in the set, or if a particle jumps onto an occupied site  then the two coalesce to one. 
\end{itemize}

\noindent
The dual process can be used to compute the state of $x$ at time $t$. The first step is to work backwards in time to find $\zeta^{x,t}_t$ the set of sites at time 0 that can influence the state of $x$ at time $t$. We note the states of the sites at time $t$ and then work up the graphical representation to determine what changes should occur at the branching points in the dual.

To prove Theorem \ref{CDlim}, Cox and Durrett \cite{CD} show that after a branching event any coalescence between the particle that branched and the two newly created particles will happen quickly, in time $O(1)$ or these particles will need time $O(n)$ to coalesce. Let $L = n^{1/d}$ be the side length of the torus. When $\lambda_n \gg n^{2/d}$ the particles will come to equilibrium on the torus before the next branching occurs in the dual, so we can forget about the relative location of the particles and we end up with an ODE limit. On the random graph, our assumption that all vertices have degree $\ge 3$ implies that the mixing time for random walks on these graphs is $O(\log n)$. Thus when $\lambda_n \gg \log n$, we have the situation that after a branching event there may be some coalescence in the dual at times $O(1)$ but then the existing particles will come to equilibrium on the graph before the next branching occurs in the dual. In both cases $\lambda_n \ll n$ is needed for the the perturbation to have a nontrivial effect. 

\begin{remark} There is no reason for having vertices of degree 0 in our graph. If $p_2>0$ and we look at the dynamics on the giant component then Theorem \ref{ODElimit} will hold if $\log^2 n \ll \lambda_n \ll$. The increase in the lower bound is needed to compensate for the fact that the mixing time for random walks on the graph is $O(\log^2 n)$. See e.g., Section 6.7 in \cite{RGD}.
\end{remark}

\subsection{Long time survival}
  
The latent voter model has two absorbing states $\equiv 1$ and $\equiv 2$. On a finite graph it is a finite state Markov chain so we know it will eventually reach one of them. However by analogy with the contact process on the torus \cite{MZd} and on power-law random graphs, \cite{MMVY}, 
this result should hold for times up to $\exp(\gamma n)$ for some $\gamma>0$. 
Unfortunately we are only able to prove survival for any power of $n$.

\begin{theorem} \label{persist}
Suppose that $\log n \ll \lambda_n \ll n$ and $t_n\to\infty$. Let $\ep > 0$ and $k<\infty$. 
There is a $T_0$ that depends on the initial density so that if $n$ is large then
with high probability
$$
|U^n(t) - 1/2| \le 5\ep \qquad\hbox{for all $t \in [T_0,n^k]$.}
$$
\end{theorem}

\begin{remark} Here and in what follows `` with high probability'' means with probability $\to 1$ as $n\to\infty$.
\end{remark}

\noindent
To prove this, we use Theorem 4.2 of Darling and Norris \cite{DN}. To state their result we need to introduce some notation. Let $\xi_t$ be a continuous time Markov chain with countable state space $S$ and jump rates $q(\xi,\xi')$. In our case $\xi_t$ will be the state of the voter model on the random graph. For their coordinate function $x: S \to \RR^d$ we will take $d=1$ and 
$$
x(\xi) = \frac{1}{n} \sum_{x\in G_n} 1_{\{ \xi_{\lambda_n t}(x)=1\} }.
$$
We are interested in proving an ODE limit for $X_t =x(\xi_t)$. To compare with the paper note that our $\xi_t$ is their $X_t$ and our $X_t$ is their 
${\bf X}_t$. 

For each $\xi\in S$ we define the drift vector
$$
\beta(\xi) = \sum_{\xi'\neq\xi} (x(\xi')-x(\xi)) q(\xi,\xi')
$$
We let $b$ be the drift of the proposed deterministic limit limit $x_t$:
$$
x_t = x_0 + \int_0^t b(x_s) \,ds.
$$
In our case $b(x) =cx(1-x)(1-2x)$. To measure the size of the jumps we let
$\sigma_\theta(x) = e^{\theta|x|} - 1 - \theta|x|$ and let
$$
\phi(\xi,\theta) = \sum_{\xi'\neq \xi} \sigma_\theta( x(\xi')-x(\xi) ).
$$

Consider the events $\Omega_0 = \{ |X_0 - x_0| \le \eta \}$,
\begin{align*}
&\Omega_1  = \left\{ \int_0^{t} |\beta(\xi_s) - b(X_s)| \, ds \le \eta \right\}, \\
&\hbox{and }\Omega_2   = \left\{ \int_0^t \phi(\xi_s, \theta) \, ds \le \theta^2At/2 \right\}.
\end{align*}

\begin{theorem} \label{DarNor}
Under the conditions above, for each fixed $t$
$$
\mathbb{P} \left( \sup_{s \le t} |X_s - x_s| > \ep \right) \le 2de^{-\delta^2/(2At)} + P( \Omega_0^c \cup  \Omega_1^c \cup  \Omega_2^c )
$$ 
\end{theorem}

\noindent
To check the conditions we note

\begin{itemize}
  \item 
We have jumps that change the density by $1/n$ at times of a Poisson processes at total rate $\le M \lambda n$, so if we let $A = 2\lambda/n$ 
then $P(\Omega_2^c) \le \exp( - c \lambda n )$. 
  \item
$P(\Omega^c_0)\le \exp(-cn)$ since we will take $\xi_0$ to be product measure and $x_0$ to be its density. 
\item
The hard work comes in estimating $P(\Omega_1^c)$, i.e., estimating the difference in the drift in the particle system from what we compute on the
basis of the current density. We do this by computing the expected value of high moments of the difference $|\beta(\xi_s) - b(X_s)|$
so we end up with estimates that 
for a fixed time are $\le n^{-m}$. By subdividing the interval into small pieces we can use the single time estimates to control the 
supremum and hence the integral but only over a bounded time interval. However this is enough since it allows us to show that when the density wanders
more than $4\ep$ away from 1/2, we can return it to within $2\ep$ with probability $n^{-m}$, and in addition never have the difference exceed $5\ep$. 
\end{itemize}

Theorem \ref{ODElimit} is proved in Section 2 and Theorem \ref{persist} in Section 3. These results hold for other voter model perturbations
such as the evolutionary games considered in \cite{CD}. However, the main obstacle to proving a general result is to find a formulation that works
well on graphs with variable degrees.
The arguments in the first proof closely parallel arguments in \cite{CD} but now use estimates for random walks on random graphs.
The keys to the second proof are results concerning the behavior of coalescing random walks (CRWs). There have been
a number of studies of the time it takes for CRWs starting from every site of a random graph to coalesce to 1. See results by Cooper et al \cite{CooperReg,Cooper} and Oliveira \cite{OliRRW,OliMF}. Here we need results about the decay of density of particles at short times. Since we are content with upper bounds the work is not hard (see Section 3.1). However, it seems difficult to prove generalizations of the results of  
Sawyer \cite{Sawyer} and Bramson and Griffeath \cite{BGcrw}  because the results on $\ZZ^d$ rely heavily on translation invariance. 

\clearp

\section{Proof of Theorem \ref{ODElimit}}

\subsection{Mixing times for random walks} 

Bounds for the mixing times come from studying the conductance
$$
Q(x,y) = \pi(x) q(x,y)
$$
where $\pi$ is the stationary distribution and $q(x,y)$ is the rate of jumping from $x$ to $y$. In the site version $q(x,y) = 1/d(x)$ while
$\pi(x) = d(x)/D$ when $y$ is a neighbor of $x$, $y\sim x$, so $Q(x,y) = 1/D$ when $y \sim x$. In the edge version, $q(x,y)=1$ if $y \sim x$, while
$\pi(x) = 1/n$ where $n$ is the number of vertices, so $Q(x,y) = 1/n$ when $y\sim x$. When the mean degree $\sum_k kp_k < \infty$, 
the two conductances are the same up to a constant. 

Define the isoperimetric constant by
$$
h = \min_{\pi(S)\le 1/2} \frac{Q(S,S^c)}{\pi(S)}  
$$
where $\pi(S) = \sum_{x\in S} \pi(x)$ and $Q(S,S^c) = \sum_{x\in S,y\in S^c} Q(x,y)$. Cheeger's inequality, see e.g. Theorem 6.2.1. in \cite{RGD}
implies that the spectral gap $\beta = 1- \lambda_1$ has
\beq
\frac{h^2}{2} \le \beta \le 2h
\label{Cheeger}
\eeq
Using Theorem 6.1.2 in \cite{RGD} we see that
\beq
\Delta(t) \equiv \max_{x,y} \left| \frac{p_t(x,y)}{\pi(y)} - 1 \right| \le \frac{e^{-\beta t} }{\pi_{min}}
\label{convbd}
\eeq
where $\pi_{min} = \min \pi(x)$.

Gkantsis, Mihail, and Saberi \cite{GMS} have shown, see Theorem 6.3.2. in \cite{RGD}:

\begin{theorem} \label{GMSth}
Consider a random graph in which the minimum degree is $\ge 3$. There is a constant $\alpha_0$ so that $h\ge \alpha_0$.
\end{theorem} 

\noindent
Combining the last result with \eqref{Cheeger}, \eqref{convbd}, and the fact that $\pi_{min} \ge 1/(C_0n)$ for large $n$, we see that 
$$
\Delta(t) \le C_0 n e^{-\gamma t} \qquad \hbox{where $\gamma = \alpha_0^2/2$.}
$$
If we let $C_1 = (6/\alpha_0^2)$ then $n$ large we have for $t \ge C_1 \log n$
\beq
\Delta(t) \le 1/n
\label{convbd}
\eeq

\subsection{Our random graph is (almost) locally a tree}

Recall that to construct our random graph we let $d_1, d_2, \ldots d_n$ be i.i.d.~from the degree distribution conditioned on $d_1+ \cdot + d_n$ to be even
and then we pair the half-edges at random. Given a vertex $x$
with degree $d(x)$, we let $y_1(x) \ldots y_{d(x)}(x)$ be its neighbors. To grow the graph we let $V_0=\{x\}$. On the first step we
draw edges from $x$ to  $y_1(x) \ldots y_{d(x)}(x)$ and let $V_1 = \{ y_1(x), \ldots, y_{d(x)}(x) \}$ which we consider to be an ordered list.
If $V_t$ has been constructed we let $x_t$ be the first element of $V_t$ and draw edges from $x_t$ to  $y_1(x_t) \ldots y_{d(x_t)}(x_t)$.
We then add the members of  $y_1(x_t) \ldots y_{d(x_t)}(x_t)$ not already in $V_t$ to it to create $V_{t+1}$. 

We stop when we have determined the neighbors of all vertices at distance $<(1/5) \log_M n$ from $x$. A simple calculation using branching processes shows that the total number of neighbors within that distance of $x$ is $\le n^{1/5} \log n$ for large n. The $\log n$ takes care of the limiting random variable. Thus in the construction we will generate $\le M n^{1/5} \log n$ connections. We say that a collision occurs at time $t$ if we connect to a vertex already in $V_t$. The probability of a collision on  single connection is $\le M n^{-4/5}\log n$. The expected number of collisions stating from any site is $\le CM n^{-3/5}\log^2 n$, so for most starting points (but not all) the graph will be a tree. To get a conclusion that applies to all starting points we note that the probability of two collisions in the construction starting from one site is 
$$
\le \binom{ CMn^{1/5}\log n }{2} (n^{-4/5}\log^2 n)^2 =O(n^{-6/5}\log^6 n)
$$ 

As we build up the graph we first find all of the neighbors of vertices at distance 1 from $x$ then distance 2, etc. Thus when a collision occurs it will connect a vertex at distance $k$ with one at distance $k$ or to one at distance $k+1$ that already has a neighbor at distance $k$. As we will explain after the next lemma, this makes very little difference.

\subsection{Results for hitting times}

\begin{lemma} \label{W1}
Once two particles are a distance $r_n=2\log_{2}\log n$ then, for large $n$, with probability $\ge 1 - 2/(\log n)^2$, 
they will reach a distance $5 r_n$ before hitting each other. 
\end{lemma}

\begin{proof} For the proof we will pretend that the graph is exactly a tree up to distance $5r_n$.
Let $Z_t$ be the distance between these two particles and let $T_m$ be the first time the distance is $m$.
Note that on each jump, with probability $p\ge 2/3$, the particles get 1 step further apart,
while with probability $\le 1/3$, the particles get one step closer. This implies that
$\phi(z)=(1/2)^z$ is a supermartingale, so
$$
\phi(r_n) \ge P_{r_n}(T_0<T_{10\log_2\log n})\phi(0) + (1-P_{r_n}(T_0<T_{10\log_2\log n}))\phi(10\log_2\log n).
$$
Rearranging gives
\begin{align}
P_{r_n}(T_0<T_{10\log_2\log n}) & \le  \frac{\phi(10\log_2\log n)-\phi(r_n)}{\phi(10\log_2\log n)-\phi(0)}
\label{L1bd} \\
&=\frac{1/(\log n)^{10}-1/(\log n)^2}{1/(\log n)^{10}-1} \sim \frac{1}{(\log n)^2}\nonumber
\end{align}
as $n \to\infty$ which proves the desired result. \end{proof}

\begin{remark}
As noted after the construction, when a collision occurs it will connect a vertex at distance $k$ with one at distance $k$ or to one at distance $k+1$ that already has a neighbor at distance $k$. In the first case at distance $k$ the comparison chain moves towards $x$ with probability $\le 1/3$, the chain stays at the same distance with probability $\le 1/3$ and moves further away with probability $\ge 1/3$. In the second case at distance $k+1$ the comparison chain moves toward the root with probability $\le 2/3$ and further away with probability $\ge 1/3$.

If we have a birth and death chain $X_n$ that jumps $p(k,k+1)=p_k$, $p(k,k)=r_k$ and $p(k,k-1)=q_k$ then 
$$
\phi(k+1)-\phi(k) = \frac{q_k}{p_k} [\phi(k)-\phi(k-1)]
$$
recursively defines a function $\phi$ so that $\phi(X_n)$ is a martingale. In our comparison chain $q_k/p_k = 1/2$ for all but one value of $k$, so 
$\phi(k)/2^{-k}$ is bounded and bounded away from 0. Thus, calculations like the one in \eqref{L1bd} will work but give a slightly larger constant.
Because of this we will avoid ugliness by assuming the graph is exactly tree like. 
\end{remark}

To prepare for the next result we need

\begin{lemma} \label{LD}
If $S_k$ is the sum of $k$ independent mean one exponentials then
$$
P( S_k \le ak ) \le \left( \frac{ae}{1+a} \right)^k
$$
\end{lemma}
\begin{remark} This holds for all $a$ but is only useful when $ae/(1+a) < 1$, which holds if $a<1/2$.
\end{remark}

\begin{proof} Let $\theta>0$ and note $\int_0^\infty e^{-\theta x} e^{-x} \, dx = 1/(1+\theta)$. Using Markov's inequlaity we have
$$
e^{-\theta a k} P( S_k \le ak ) \le (1+\theta)^k
$$
Taking $\theta=1/a$ and rearranging gives the desired result. 
\end{proof}

\begin{lemma} \label{W2}
Suppose two particles are a distance $r_n=2\log_{2}\log n$. Then with high probability the two particles will not collide by time $\log^2 n$.
\end{lemma}

\begin{proof} A particle must make $4r_n$ jumps to go from distance $5r_n$ to $r_n$. Since jumps occur at rate 1 in the site model
and at rate $\le M$ in the edge model, the last lemma implies that the probability of $k=r_n$
jumps in time $\le ar_n/M$ is  
$$
\le (ae)^{r_n} \le 1/(\log^3 n)
$$
for large $n$ if $a$ is small enough. If we make $2M (\log^2n)/ar_n$ attempts to reach 0 before $5r_n$ starting from  $r_n$
then Lemma \ref{W1} implies that with high probability we will not be successful, while the last bound implies implies that  
this number of attempts will take time $\ge 2 \log^2 n$ with high probability. 
\end{proof} 

\begin{lemma} \label{W3}
Suppose two particles are a distance $r_n=2\log_{2}\log n$ and let $s_n/n \to 0$. 
Then with high probability the two particles will not hit by time $s_n$.
\end{lemma}

\begin{proof} Lemma \ref{W2} takes care of times up to $\log^2 n$. The result in \eqref{convbd} implies that 
if $n$ is large then for $t \ge \log^2 n$, $p_t(x,y) \le 2/n$.
Summing we see that if the two particle move independently the expected amount of time the 
two particles spend at the same site at times in $[\log^2 n, s_n]$ is $\le 2s_n/n \to 0$.
Since the jump rates are bounded above this implies the desired result.
\end{proof}

Later we will need the following generalization of Lemma \ref{W2}. Let $x$ and $y$ be adjacent sites on the graph.
We say that the walks starting at $x$ and $y$ do not $r$-localesce if they do not hit before one of them exits the
ball of radius $r$.

\begin{lemma} \label{Wlc}
If $x$ and $y$ do not $9 \log_2\log n$-localesce then with high probability they do not hit by time
$\log^2 n$. 
\end{lemma}

\begin{proof} Let $A_n$ be the event that the walks starting from $x$ and $y$ the two particles get to a distance $2 \log_2\log n$ before they hit
and let $B_n$ be the event that they do not $9 \log_2\log n$-localesce. Since the distance
between the walks increases by 1 with probability 2/3 and decreases by 1 with probability 1/3 then with high probability
$A_n$ will occur before the total number of steps made by either particle $\le 8 \log_2\log n$. Thus $P(B_n \cap A_n^c)$ is small
and the desired result follows from Lemma \ref{W2}.
\end{proof}

\subsection{Results for the dual process}

In this section we will consider the dual process on its original time scale, i.e., jumps occur ate rate $O(1)$. 
In either version of the model, the rate at which branching occurs is $\le L/\lambda$ where $L=M^2$. (Here we are using the fact
in the edge model the degree is bounded.) Let $R_n$ be time of the $n$th branching. If $t_n=c_2\log n$ for some constant $c_2 > 0$ then
$$
P(R_1\leq t_n) \rightarrow 0 \indent as \ n \rightarrow \infty
$$
Let $N(t)$ be the number of branching events by time $\lambda t$. Comparing with a branching process we have $EN(t) \le e^{Lt}$.
The expected number of branchings in the interval $[\lambda t - t_n,\lambda t]$ is $\le e^{Lt}(c_2\log n)/\lambda$ so 
as $n \rightarrow \infty$,
\begin{equation}
P(\lambda t-R_{N(t)}\leq t_n) \rightarrow 0
\label{endspace}
\end{equation}
In the next three results $C_1$ is the constant defined in \eqref{convbd} and we make the following assumption:

\mn
(A3) Suppose there are $k$ particles in the dual at time 0, and each pair are separated by a distance $r_n =2\log_{2}\log n $.

\begin{lemma} \label{L1}
Suppose that at time 0, the first particle encounters an branching event. By time $C_1\log n$, there may be coalescences between new born particles or with their parent, but with high probability there will be no other coalescences.
\end{lemma}

\begin{proof} This follows from Lemma \ref{W2}.
\end{proof}

\begin{lemma} \label{L2}
At time $C_1\log n$ all the particles are almost uniformly distributed on the graph with the bound on the total variation distance uniform over
all configurations allowed by (A3).
\end{lemma}

\begin{proof} This follows from \eqref{convbd}.
\end{proof}

\begin{lemma} \label{L3}
After time $C_1\log n$, with high probability there is no coalescence between particles before the next branching event, and right before the next branching event, all the particles are $r_n$ apart away from each other.
\end{lemma}

\begin{proof} The claim about coalescence follows from Lemma \ref{W3}. The branching time
is random but it is independent of the movement of the particles, so the result about the separation between particles follows from \eqref{convbd}.
\end{proof}

Together with \eqref{endspace}, Lemma \ref{L3} implies that there is no coalescence in the dual $[R_{N(t)}, \lambda t]$ and particles are at least $r_n$ apart right before $R_{N(t)}$. According to Lemma \ref{L2}, the coalescences between new born particles and their parents can only happen before $R_{N(t)}+C_1\log n$, with no other coalescences. Lemma \ref{L2} tells us at times $\ge R_{N(t)}+C_1\log n$, all the particles are almost uniformly distributed over the graph. Thus when we feed values into the dual process to begin to compute the state of $x$ at time $t$ the values are independent and equal to 1 with probability $u$. 

\begin{lemma}
$EU^n(t)$ converges to a limit $u(t)$.
\end{lemma}

\begin{proof} Let $Z(s)$, $s \le t$ be the number of particles in the dual process, when we impose the rule that the number of particles is not increases until 
time $(C_1 \log n)/\lambda$ after a branching event. Our results imply that $Z(s)$ converges to a branching process. The last result shows that
when when we use the dual to compute the state of $x$ at time $t$ we put independent and identically distributed values at the $Z(t)$ sites. 
The result now follows from results in \cite{CDP}.
\end{proof}

\begin{lemma} 
$U^n(t) - EU^n(t)$ converges in probability to 0.
\end{lemma}

\begin{proof} 
It follows from Lemma \ref{W3} that if $|x-y| > r_n$ then there will be no collisions between particles in the dual processes starting from $x$ and $y$,
and hence the values we compute for $x$ and $y$ are independent. The result now follows from Chebyshev's inequality.
\end{proof}

\subsection{Computation of the reaction term}

The final step is to show that $u(t)$ satisfies the differential equation. 
On $\ZZ^d$ if  $\nu_u$ the voter model stationary distribution with density $u$ and $v_1$ and $v_2$ are randomly chosen neighbors of $x$ then 
$$
\langle h_{1,2}(x,\xi)\rangle_u = \nu_u( \xi(x) =1, \xi(v_1)=2 \hbox{ or } \xi(v_2) = 2 )
$$
The right-hand side can be computed using the duality between the voter model and coalescing random walk. Following the approach in Section 4 of \cite{RDEG}
if we let $p(x|y|z)$ be the probability the random walks starting from $x$, $y$, and $z$ never hit and $p(x|y,z)$ be the probability $y$ and $z$ coalesce but
don't hit $x$ then
$$
\nu_u( \xi(x) =1, \xi(y)=2 \hbox{ or } \xi(z) = 2 ) = p(x|y|z)u (1 - u^2) + p(x|y,z) u(1-u)
$$
Using this identity we can compute the reaction term defined in \eqref{react} 
\begin{align}
\phi(u) & = \langle h_{2,1}(x,\xi) - h_{1,2}(x,\xi) \rangle_u
\nonumber\\
&= p(x|v_1|v_2)(1-u) (1 - (1-u)^2) + p(x|v_1,v_2) u(1-u) \nonumber\\
&\qquad - [p(x|v_1|v_2)u (1 - u^2) + p(x|v_1,v_2) u(1-u)] \label{Zdreact} \\
& = p(x|v_1|v_2) [ (1-u)u(2-u) - u(1-u)(1+u) ] \nonumber\\
& = p(x|v_1|v_2) u(1-u)(1-2u) \nonumber
\end{align}

The computations for the random graph are similar but in that setting we have to take into account the degree of $x$ and what the graph looks like
locally seen from $x$. Let $q_k$ be the size-biased distribution $k p_k/\mu$ where $\mu=\sum_k p_k$ is the mean degree.
Let $\PP_{k}$ be a Galton Watson tree in which the root has degree $k$ and the other vertices have $j$ children with probability $q_{j+1}$.

In the site version a dual random walk path will spend a fraction $\pi^s(k) = q_k$ at vertices with degree $k$ so 
$$
\langle h^s_{2,1} - h^s_{1,2} \rangle_u = \sum_k q_k \PP_{k}(x|v_1|v_2) u(1-u)(1-2u)
$$
where $v_1$ and $v_2$ are randomly chosen neighbors of the root. In the edge version $\pi^e(k) = p_k$ so
$$
\langle h^e_{2,1} - h^e_{1,2} \rangle_u = \sum_k p_k \PP_{k}(x|y|z) u(1-u)(1-2u)
$$

\clearp

\section{Proof of Theorem \ref{persist}}

Recall that the density in the time-rescaled latent voter model is given by:
$$
X_t=x(\xi_{\lambda t}) = (1/n) \sum_{x\in G_n}\mathds{1}(\xi_{\lambda t}(x)=1).
$$ 
To complete the proof of Theorem \ref{persist} using the result of Darling and Norris \cite{DN} given in Theorem \ref{DarNor} we need to estimate the
probability of
\beq
\Omega_1 = \left\{ \int_0^t |\beta(X_s) - b(X_s)| \, ds  \le \eta \right\}
\label{omega1}
\eeq
where $\beta(\xi) = \sum_{\xi'\neq \xi} (x(\xi')-x(\xi)) q(\xi,\xi')$ is the drift in the particle system and $b(u) = c_pu(1-u)(1-2u)$ 
is the drift in the ODE. 

To begin to do this, we define $\tilde{\xi}(s)$ to be the same as $\xi(s)$ for time $s\leq \lambda t-C_1 \log n$, while on the time interval $(\lambda t-C_1 \log n, \ \leq \lambda t ]$, $\tilde{\xi}$ only follows the paths from voter events of $\xi$, ignoring those from branching events. Let 
$$
\tilde{X}_t=x(\tilde\xi_{\lambda t})= \frac{1}{n} \sum_{x \in G_n} 1_{ \{\tilde{\xi}_{\lambda t}(x)=1 \} }
$$ 
be the density of this new process $\tilde{\xi}$. In order to determine $\tilde{\xi}_{\lambda t}$, we run the coalescing random walks backward in time, starting from time $\lambda t$ and stopping at time $\lambda t-C_1\log n$. Since $C_1\log n/\lambda \to 0$, then with high probability the dual random walk starting from a site $x$ will not encounter any branching event in time $(\lambda t-C_1 \log n, \ \lambda t]$, so $\tilde X_t$ will be close to $X_t$ with high probability.

Let $\tilde{u} = x(\xi_{t - (C_1\log n)/\lambda})$.  Our first step toward bounding $P(\Omega_1^c)$ is.

\begin{lemma} \label{mainL}
 Suppose $\log n \ll \lambda_n \ll n$ and $m>0$. There is a $C_m$ so that for any $\delta>0$ if $n \ge n_0(k)$
\begin{equation}
\mathbb{P}\left(|\tilde{X_t} - \tilde{u}|>\epsilon |{\cal F}_{t-(C_1\log n)/\lambda}\right) \leq n^{-m} + \frac{C_m}{\ep^{2m} n^{m(1-\delta)}} 
\end{equation}
\end{lemma}

 We say that two sites at time $t$ are in the same cluster if their random walks have coalesced. The state of the process at time $\lambda t - C_1\log n$ is close to a voter model equilibrium with density $U(\lambda t - C_1\log n)$. By Lemma \ref{L2}, particles that have not coalesced by time $\lambda t - C_1\log n)$ are separated by a large distance on the graph, so if we look at the states of a fixed finite number of clusters then in the limit as $n\to\infty$ they are independent. 

Let $N_1, \ldots N_k$ be the sizes of clusters at time $\lambda t$, 
and let $\tilde{\eta}_1,\ldots, \tilde{\eta}_k$ are i.i.d with $\mathbb{P}(\tilde{\eta}_j=1)=\tilde{u}$.
\begin{align}
&n \tilde{X}_t =  \sum_{x}1_{ \{\tilde{\xi}_{\lambda t}(x)=1\} }=\sum_{j=1}^kN_j\tilde{\eta}_j
\label{sumcl}\\ 
&n\tilde{u}=\sum_{j=1}^k N_j\tilde{u}=\sum_{j=1}^k N_j\mathbb{E}\tilde\eta_j
\label{meancl}
\end{align}
If we can get a good bound on $N_{max} = \max N_i$ then we can estimate $\tilde{X}_t - \tilde{u}$.

\subsection{Bounds on cluster sizes}

Let $N_x(s)$ be the size of the cluster containing the particle that started at $x$ at time $t$ when we run the coalescing random walk to time $t-s$.
We begin by considering the edge model.

\begin{lemma} \label{meanc}
If  $s \ge 1/2M$ then $\mathbb{E}(N_x(s)-1) \le 4Mes$.
\end{lemma}

\begin{proof}  
Let $y\neq x$ and $W^y$ be the edge random walk starting from $y$. Noting that when $W^x$ and $W^y$ hit, they stay together for a time $\geq 1/2M$ with probability $e^{-1}$ gives 
$$
\mathbb{P}(W^x \ and \ W^y \ hit \ by \ time \ s)\times\frac{1}{2Me}
\leq \int_0^{s+1/2M}\sum_z p_r(x,z)p_r(y,z) \, dr
$$
Since the edge random walks are reversible with respect to the uniform distribution, the transition probability is symmetric
\begin{align}
\int_0^{s+1/2M}\sum_z p_r(x,z)p_r(y,z)\, dr  &= \int_0^{s+1/2M}\sum_z p_r(x,z)p_r(z,y) \, dr\\
& = \int_0^{s+1/2M} p_{2r}(x,y) \, dr 
\end{align}
Using this we have
$$
EN_x(s) = \sum_y \mathbb{P}(W^x \ and \ W^y \ hit \ by \ time \ s) \le 2Me  \int_0^{s+1/2M}  \, dr \le 4Mes  
$$
where in the last step we have used $s \ge 1/2M$
\end{proof}

Our next step is to bound the second moment of $N_x(t)$. 

\begin{lemma} \label{m2c}
If  $s \ge 1/2M$ then $\mathbb{E}(N_x(s)-1)(N_x(s)-2) \le 3(4Mes)^2$.
\end{lemma}

\begin{proof}
We begin by observing that
$$
\EE (N_x(s)-1)(N_x(s)-2) = \sum_{x_1,x_2} P(x_1,x_2 \in N_x(s)).
$$
where the sum is over $x_i \neq x$ and $x_1\neq x_2$. 
We first consider the case in which $x$ and $x_1$ are the first to collide, and we bound
$$
\sum_{x_1,x_2,y,z} \int_0^{s+1/2M} p_r(x,y)p_r(x_1,y) p_r(x_2,z) P( z \in N_{y,r}(s) ) \, dr
$$
where $N_{y,r}(s)$ is the cluster at time $s$ of the random walk that starts at $y$ at time $r$. 
As in the previous proof $2Me$ times this quantity will bound the desired hitting probability. By symmetry 
$\sum_{x_2} p(x_2,z) = \sum_{x_2}p(z,x_2) = 1$. Using Lemma \ref{meanc}
$$
\sum_z P( z \in N_{y,r}(s) ) \le 4Mes
$$
Using reversibility we can write what remains of the sum as
\beq
\sum_{x_1,y}  \int_0^{s+1/2M} p_r(x,y)p_r(y,x_1) \, dr = \sum_{x_1} \int_0^{s+1/2M} p_{2r}(x,x_1) \, dr \le 2s
\label{c1end}
\eeq

The second case to consider is when $x_1$ and $x_2$ are the first to collide, and we bound
$$
\sum_{x_1,x_2,y,z} \int_0^{s+1/2M} p_r(x_1,y)p_r(x_2,y) p_r(x,z) P( z \in N_{y,r}(s) ) \, dr
$$
Using symmetry $p_r(x_1,y)p_r(x_2,y)= p_r(y,x_1)p_r(y,x_2)$ then summing over $x_1,x_2$ we have
$$
\le \sum_{y,z} \int_0^{s+1/2M} p_r(x,z) P( z \in N_{y,r}(s) ) \, dr  
$$
We have $P( z \in N_{y,r}(s) ) = P( y \in N_{z,r}(s) )$ because either event says $y$ and $z$ coalesce in $[r,s]$,
so summing over $y$ and using Lemma \ref{meanc} the above is
\beq
\le (4Mes) \sum_z  \, \int_0^{s+1/2M} p_r(x,z) \, dr  \le (4Mes) \cdot 2s \label{c2end}
\eeq
Combining our calculations proves the desired result. \end{proof}

\begin{lemma} \label{mkc}
If $s \ge 1/2M$ then $\EE[(N_x(s)-1)\cdots (N_k(s)-k)] \le C_k (4Mes)^k$ and hence
$$
\EE N_x^m(s) \le C_{m,M}(1+s)^m
$$
\end{lemma} 

\begin{proof}
The second result follows easily from the first since
$$
x^m = 1 + \sum_{k=1}^m c_{m,k} (x-1)\cdots (x-k)
$$
The first case is
$$
\sum_{x_1,\ldots,x_k,\atop y,z_1,\ldots z_{k-1}} 
\int_0^{s+1/2M} p_r(x,y)p_r(x_1,y) p_r(x_2,z_1)\ldots p_r(x_k,z_{k-1}) P( z_1,\ldots z_{k-1} \in N_{y,r}(s) ) \, dr
$$
Using symmetry and summing over $x_2, \ldots, x_{k}$ removes the  $p_r(x_2,z_1)\ldots p_r(x_k,z_{k-1})$ from the sum. 
Next we  sum over $z_1,\ldots z_{k-1}$ (which are distinct) and use induction to bound the sum by $C_{k-1}(4Mes)^{k-1}$.
Finally we finish up by applying \eqref{c1end}.

The second case is
\begin{align*}
\sum_{x_1,\ldots,x_k,\atop y,z_1,\ldots z_{k-1}}
\int_0^{s+1/2M} p_r(x_1,y)p_r(x_2,y) p_r(x_3,z_1)\ldots p_r(x_k,z_{k-2})&\\
 p_r(x,z_{k-1}) P( z_1,\ldots z_{k-1} \in N_{y,r}(s) ) \, dr&
\end{align*}
Using symmetry and summing over $x_1,\ldots, x_k$ removes the 
$$
p_r(x_1,y)p_r(x_2,y)p_r(x_3,z_1) \ldots p_r(x_k,z_{k-2}). 
$$
As in the previous proof
$P( z_1,\ldots z_{k-1} \in N_{y,r}(s) ) = P( z_1,\ldots z_{k-2}, y \in N_{z_{k-1},r}(s) )$, so summing over $z_1,\ldots, z_{k-2}, y$ and 
using induction we can bound the sum by $C_{k-1}(4Mes)^{k-1}$. Finally we finish up by applying \eqref{c2end} with $z=z_{k-1}$
\end{proof}

\begin{remark}
To extend to the site case where we do not have symmetry, we note that reversibility of this model with respect to $\pi(y)=d(y)/D$  implies
$$
p_r(y,z) \le d(y)p_r(y,z) = d(z)p_r(z,y) \le M p_r(z,y)
$$ 
so the proof works as before but we accumulate a factor of $M$ each time we use symmetry.
\end{remark}

Now we are ready to give an upper bound on the size of the maximal cluster $N_{max}(t)$ at time $\lambda t$.
Here and for the rest of the proof of Lemma \ref{mainL}, we only use moment bounds so the proof is the same for the edge and site models 

\begin{lemma} \label{Nmax}
Let $\delta>0$ and $m < \infty$. If $t \le \log^2 n$ Then for large $n$
$$
P( N_{max}(t) > n^\delta) \le n^{-m}
$$ 
\end{lemma}

\begin{proof} By Chebyshev's inequality
$$
n^{\delta k} P( N_x(t) > n^\delta ) \le C_{k,M} (1+t)^k
$$
If we pick $k > (m+1)/\delta$ then 
$$
P\left( \max_x N_x(t) > n^\delta \right) \le \frac{n}{n^{k\delta}} C_{k,M}(2 \log^2 n)^k = o(n^{-m})
$$
which proves the desired result. 
\end{proof} 

\clearp

\subsection{Moment estimates}

\begin{proof}[Proof of Lemma \ref{mainL}]  
Based on Lemma \ref{Nmax} we let 
\begin{equation}
A_n=\{\xi: N_{max}\leq n^\delta \}
\label{defAn}
\end{equation}
To simplify formulas, let $Y_j=N_j\mathds{1}(\tilde{\xi}_j=1)-N_j\tilde{u}$. Note that $|Y_j|\leq N_j$ and $Y_1,...,Y_k$ are independent with mean 0. 
If there are $k$ clusters 
$$
\sum_x 1_{\{\tilde{\xi}(x)=1\}}-n\tilde{u} = \sum_{j=1}^k Y_j
$$
so we have 
\beq
\mathbb{E}\left[ \left( \sum_{x} 1_{ \{ \tilde{\xi}(x)=1 \} } -n \tilde{u}\right)^{\sqz 2m} ; A_n\right]
= \mathbb{E} \left(\mathbb{E}\left[\left.\left(\sum_{j=1}^k Y_j\right)^{2m}\right| N_1,...,N_k\right];A_n \right)
\label{form39}
\eeq
Writing $\bar\EE$ for the expectation conditional on $N_1,...,N_k$, we will show there is a constant $C_m$ so that
\begin{equation}
\bar\EE\left(\sum_{j=1}^k Y_j\right)^{2m}\leq C_m\left( N_{max}n\right)^{2m}
\end{equation}
Let $l$ denote the number of different $Y_i$, and let $\mathcal{I}_l$ be the set of all possible powers 
$$
\left\{\left(k_1,\ k_2, ..., k_l\right): k_1+\cdots+k_l=2m \ and \ 2\leq k_1 \leq ...\leq k_l\right\}.
$$ 
We restrict to $k_i \ge 2$ since if there is a $k_j=1$ we will have  $\mathbb{E}Y^{k_1}_{i_1}\cdots Y^{k_l}_{i_l}=0$. In the following, the subscript $*$ in $\sum_*$ means all the indices
+ $i_1,...,i_l$ are distinct.
$$
\bar\EE\left( \sum_{j=1}^k Y_j \right)^{\sqz 2m} 
\le \sum_{l\leq m}\sum_{\mathcal{I}_l}\sum_*\mathbb{E}|Y_{i_1}^{k_1}\cdots Y_{i_l}^{k_l}|
 \leq \sum_{l\leq m } \sum_{\mathcal{I}_l}\sum_*N_{i_1}^{k_1}\cdots N_{i_l}^{k_l}
$$
Note that for any fixed $(k_1, ..., k_l)\in\mathcal{I}_l$, we can always find a $(\alpha_1,...,\alpha_l)$ such that $1\leq\alpha_i < k_i, i=1,2,3,...,l$ and $\alpha_1+\cdots+\alpha_l=m$. Now factoring $N_{i_j}^{\alpha_{j}}$ out from $N_{i_j}^{k_{j}}$ and using $N_{i_j}\leq N_{max}$, we have
\beq
\bar\EE\left(\sum_{j=1}^k Y_j\right)^{\sqz 2m} 
\leq \sum_{l\leq m} N^m_{max}\sum_{\mathcal{I}_l}\sum_* N_{i_1}^{k_1-\alpha_1}\cdots N_{i_l}^{k_l-\alpha_l}
\label{sum43}
\eeq
Since
$$
\sum_* N_{i_1}^{k_1-\alpha_1}\cdots N_{i_l}^{k_l-\alpha_l}\leq(N_1+\cdots+N_k)^{\sum_{i=1}^l (k_i-\alpha_i)}=n^m
$$
\eqref{sum43} implies that
\beq
\bar\EE\left(\sum_{j=1}^k Y_j \right)^{\sqz 2m}\leq \sum_{l\leq m }N_{max}^m \sum_{\mathcal{I}_l}(N_1+\cdots+N_k)^m \leq C_m (N_{max}n)^m
\label{form46}
\eeq

Since we are restricting to $A_n$, according to \eqref{defAn}, \eqref{form39} and \eqref{form46},
$$
 \mathbb{E}\left[\left(\sum_{x}1_{\{ \tilde{\xi}(x)=1\} } -n\tilde{u}\right)^{2m} ; A_n\right]
\leq C_m n^{(1+\delta) m}
$$
Using this and Markov's inequality, 
$$
\mathbb{P}\left(|\tilde X_t - \tilde{u})|>\epsilon, A_n \right)
\leq \frac{\mathbb{E}\left[\left(\sum_{x}1_{\{ \tilde{\xi}(x)=1\} } -n\tilde{u}\right)^{2m} ; A_n\right]}
{\left(\epsilon n\right)^{2m}}\\
\leq \frac{C_m n^{(1+\delta)m}}{(\epsilon n)^{2m}}
$$
Combining this with Lemma \ref{Nmax} completes the proof of Lemma \ref{mainL}
\end{proof}

\subsection{Bounding the drift}

The drift 
\begin{align*}
\beta(\xi_t)=\frac{1}{n}\sum_{x\in G_n}\sum_{y\sim x} \sum_{z\sim x, z \neq y}
& [1_{ \{ \xi_{\lambda t}(x)=2, \ \xi_{\lambda t}(y)=1\hbox{ or }\xi_{\lambda t}(z)=1 \} }\\
&-1_{ \{ \xi_{\lambda t}(x)=1, \ \xi_{\lambda t}(y)=2 \hbox{ or }\xi_{\lambda t}=2 \} }]
\end{align*}
We want to show 

\begin{lemma} \label{driftbd}
There is a constant $C_m$ if $\delta>0$ and $n \ge n_0(\delta)$
\begin{equation}
\mathbb{P}(|\beta(\xi_t)-b(X_t)| \geq \epsilon |{\cal F}_{t- (C_1\log n)/\lambda} )\leq \frac{C_m}{\ep^{2m} n^{m(1-\delta)}}
\end{equation}
\end{lemma}

\begin{proof} 
If we let $\mathds{1}(x|y|z)$ is the indicator function of the event that the dual random walks starting from $x$, $y$, and $z$ at time $t$ do not hit by time
$t - C_1(\log n)/\lambda$ and $p(x|y|z) = \EE\mathds{1}(x|y|z)$ then 
\begin{align}
\mathbb{E}[\beta(\xi_t)|{\cal F}_{t-C_1(\log n)/\lambda}]& \approx
\frac{1}{n}\sum_{x\in G_n}\sum_{y\sim x} \sum_{z\sim x, z\neq y}\mathds{1}(x|y|z) \tilde{u}(1-\tilde{u})(1-2\tilde{u}) 
\label{betaeq}\\
b(X_{t-C_1(\log n)/\lambda})& \approx 
\frac{1}{n}\sum_{x\in G_n}\sum_{y\sim x} \sum_{z\sim x, z\neq y}p(x|y|z) \tilde{u}(1-\tilde{u})(1-2\tilde{u}) 
\label{beq}
\end{align}
where $\approx$ means that the probability the difference $>\ep$ tends to 0 as $n\to\infty$.

The random variables $\mathds{1}(x|y|z)$ are dependent if the triples $(x,y,x)$ and $(x',y',z')$ overlap or if the associated random walks
coalesce. To simplify things we will let $\hat{\mathds{1}}(x|y|z)$ be the event none of the walks $r$-localesce (i.e., the pair collides before either of them exits $B(x,r)$. Lemma \ref{Wlc} implies that if we pick $r=9 \log_2 \log n$ then with high probability the two walks will not hit by time $\log^2 n$.

Imitating the previous proof we will let
$$
Y_{x,y,z} = \hat{\mathds{1}}(x|y|z) - \hat p(x|y|z)
$$
where $\hat p(x|y|z) = \EE (\hat{\mathds{1}}(x|y|z)$ and then compute $\EE( \sum_{x,y,z} Y_{x,y,z})^{2m}$ where the sum is over 
$x\in G_n$ and neighbors $y$, $z\neq y$ of $x$. 

\begin{lemma} \label{EY2m}
$E \left( \sum_{x,y,z} Y_{x,y,z} \right)^{2m} \le (n M^2)^m (\log n)^{27m \log_2 M}$.
\end{lemma}

\begin{proof}
The sum has $K = \sum_x d(x)(d(x)-1)$ terms. The $2m$th moment of the sum has terms of the form. 
$$
Y_{x_1,y_1,z_1} \cdots  Y_{x_{2m},y_{2m},z_{2m}} 
$$
If some $x_i$ has distance $3r$ from all of the other $x_j$ then $Y_{x_i,y_i,z_i}$ is independent of the product of the rest of
the random variables and the expected value is 0.  

Suppose now that for each $x_i$ there is at least one $x_j$ that is within distance $3r$. Create a graph $D$ (for dependency)
where there is an edge between $i$ and $j$ if $d(x_i,x_j) < 3r$. Let $\kappa$ be the number of components in the graph.
Our condition implies $\kappa \le m$. Since degree of each vertex in $G_n$ is $\le M$ the number of vertices within distance $27 \log_2 \log n$ 
of a vertex is 
$$
\le L \equiv M^{27 \log_2 \log n} = 2^{27 (\log_2\log n) \cdot \log_2 M} = (\log n)^{27 \log_2 M}
$$  
Thus the number of terms associated with graphs with $\le m$ components is
$$
\le (n M^2)^m L^m
$$
Since $\EE|Y_{x_1,y_1,z_1} \cdots  Y_{x_{2m},y_{2m},z_{2m}}| \le 1$ the desired result follows.
\end{proof}

Lemma \ref{EY2m} implies
$$
P\left( \left| \sum_{x,y,z} Y_{x,yz} \right| > \ep n \right) \le \frac{M^{2m}(\log n)^{27m \log_2 M}}{\ep^{2m} n^m}
$$
Using this with \eqref{betaeq}, \eqref{beq}, and Lemma \ref{mainL} gives Lemma \ref{driftbd}.
\end{proof}   

Taking expected value and setting $\delta=1/2$  we have shown that 
$$
\mathbb{P}(|\beta(\xi_t)-b(X_t)| \geq \epsilon |{\cal F}_{t- (C_1\log n)/\lambda} )\leq \frac{C_{m,\ep}}{n^{m/2}}
$$
To extend this to bound the probability of 
$$
\Omega_1^c = \left\{ \int_0^t |\beta(X_s)-b(X_s)| \, ds \ge \eta \right\}
$$
we subdivide the interval $[0,t]$ into subintervals of length $1/\lambda n^{1/2}$. Within each interval the probability that
more than $2n^{1/2}$ sites will flip is $\le \exp(-c\sqrt{n})$. From this it follows that when $\eta = 2t\ep$
\beq
P(\Omega_1^c) \le t\lambda n^{1/2} \left[ \frac{C_{m,\ep}}{n^{m/2}} + \exp(-c\sqrt{n}) \right]
\label{om1bd}
\eeq

\subsection{Iteration argument}

The last bound only works for fixed $t$. To get long time survival we will iterate. Let
$$
T_0 = \inf\{ t: |x_t - 1/2| < \ep \}
$$
and note that this is not random. Theorem \ref{DarNor} implies that at this time $|X_t-1/2| \le 2\ep$ with very high probability,
i.e., with an error of less that $C n^{-(m-1)/2}$ Let 
$$
T_1 = \inf\{ t > T_0 : |X_t - 1/2| \ge 4\ep \}
$$ 
and note that on $[T_0,T_1]$ we have $|X_t - 1/2| \le 4\ep$.  
There is a constant $t_0$ so that if $x(0)=1/2 + 4\ep$ or $x(0)=1/2 - 4\ep$ then
$|x(t_0)-1/2| \le \ep$. Let $S_1=T_1+t_0$. Theorem \ref{DarNor} implies that with high probability $|X(S_1)-1/2| \le 2\ep$ and
$X_t -1/2| \le 5\ep$ on $[T_1,S_1]$. For $k \ge 2$ let 
$$
T_k = \inf\{ t > S_{k-1} : |X_t - 1/2| \ge 4\ep \}\quad\hbox{and}\quad S_k=T_k+t_0. 
$$
We can with high probability iterate the construction $n^{(m-2)/2}$
times before it fails. Since each cycle takes at least $t_0$ units of time, the proof of Theorem \ref{persist} is complete.

\clearp

\end{document}